\documentclass[sn-mathphys,Numbered,11pt]{sn-jnl}% Math and Physical Sciences Reference Style

\usepackage{graphicx}%
\usepackage{multirow}%
\usepackage{amsmath,amssymb,amsfonts}%
\usepackage{amsthm}%
\usepackage{mathrsfs}%
\usepackage[title]{appendix}%
\usepackage{xcolor}%
\usepackage{textcomp}%
\usepackage{manyfoot}%
\usepackage{booktabs}%
\usepackage{algorithm}%
\usepackage{algorithmicx}%
\usepackage{algpseudocode}%
\usepackage{listings}%
\usepackage{amssymb}
\usepackage{amsmath}
\usepackage{enumerate}
\usepackage{amsfonts}

 \usepackage[cp1250]{inputenc}
\usepackage[OT4]{fontenc}
\theoremstyle{thmstyleone}%
\newtheorem{theorem}{Theorem}%  meant for continuous numbers
\newtheorem{observation}[theorem]{Observation}% 
\newtheorem{lemma}[theorem]{Lemma}

{\theoremstyle{definition}
\newtheorem{remark}{Remark}%
\newtheorem{example}{Example}}

\newcommand{\mcl}{\mathcal}

\newcommand{\eps}{\varepsilon}

\newcommand{\N}{\mathbb{N}}
\newcommand{\R}{\mathbb{R}}
\newcommand{\PP}{\mathbb{P}}

\newcommand{\B}{\mathcal{B}(\R)}
\newcommand{\BB}{\mathcal{B}(\R^n)}
\newcommand{\M}{\mathcal{M}^1(\R)}
\newcommand{\KK}{\mathcal{K}}

\newcommand{\supp}{\operatorname{supp}}
\newcommand{\cov}{\operatorname{cov}}
\newcommand{\cor}{\operatorname{cor}}
\newcommand{\diam}{\operatorname{diam}}
\begin{document}
%\bibliography{sn.bib}
%\bibliographystyle{apalike}
\title[Article Title]{ Conditional uncorrelation equals independence}


\author{\fnm{Dawid} \sur{Tar{\l}owski}}\email{dawid.tarlowski@gmail.com; dawid.tarlowski@im.uj.edu.pl}%ORCID 0000-0002-6824-4568



\affil{\orgdiv{Faculty of Mathematics and Computer Science}, \orgname{Jagiellonian University}, \orgaddress{\street{\L ojasiewicza 6}, \city{Krak\'ow}, \postcode{30-348}, \country{Poland}}}

%\affil[2]{\orgdiv{Department}, \orgname{Organization}, \orgaddress{\street{Street}, \city{City}, \postcode{10587}, \state{State}, \country{Country}}}

%\affil[3]{\orgdiv{Department}, \orgname{Organization}, \orgaddress{\street{Street}, \city{City}, \postcode{610101}, \state{State}, \country{Country}}}


\abstract{We show that the stochastic independence of real-valued random variables is equivalent to the conditional  uncorrelation, where  the conditioning takes place over the Cartesian products of  intervals. Next, we express the mutual independence in terms of the conditional correlation matrix. Our results extend the results of Jaworski et al. (Electron. J. Stat., 18(1), 653-673, 2024), which are based on the copula functions and assume the existence of the joint density of the variables. We relax this assumption and show that the independence characterization via conditional uncorrelation is valid in full generality - that is, for all kinds of random variables and any dependencies between them. Additionally, we analyse the assumptions under which the independence is determined by the local uncorrelation. The measure-theoretic methodology we present uses the Radon-Nikodym derivative to reduce the multidimensional characterization problem to the simple one-dimensional conditioning.  To demonstrate the potential usefulness  of the presented results,  various numerical examples are presented. }% including the areas of optimization theory and trol theory.}




\maketitle

\section{Introduction}\label{sec1}

Various measures  of dependence between random variables are an important part of probability and statistics. Probably the most recognized  metric is Pearson's linear correlation coefficient, designed to measure linear dependencies in various contexts, \cite{LN}. Naturally, Pearson's correlation may fail to detect nonlinear dependencies between variables, and thus,  while it is well known that stochastically independent random variables  $X$ and $Y$ are  uncorrelated in the sense $E[XY]=E[X]\cdot E[Y]$,  this implication rarely may be reversed, \cite{David}. More, even the uncorrelation of all the higher moments of $X$ and $Y$ is not enough to force the independence, \cite{DePaula}. In this paper, following \cite{JJP}, we show that the independence of random variables $X$ and $Y$ may be completely characterised by the conditional uncorrelation of the form 
\begin{equation}\label{Complete}
E[XY|U]=E[X|U]\cdot E[Y|U],
\end{equation} where $U=\{X\in A, Y\in B\}$ and $A$, $B$ are either bounded intervals or half-bounded intervals (notation and the details will be introduced in Section \ref{S2}).  
Various types of conditional correlation have been analysed in the literature,  \cite{Akeman},\cite{Baba},\cite{JJP},\cite{Lawrance},\cite{LT}, and paper \cite{JJP} has shown that the conditional 
uncorrelation of the form \eqref{Complete} characterizes independence completely. However, the results of \cite{JJP} are based on the assumption that the vector $(X,Y)$ has a joint probability density 
which is fully supported on the cube $[-1,1]^2$. Our paper addresses those limitations and shows that the characterization of independence via the conditional uncorrelation of the form \eqref{Complete} 
holds true in full generality, that is,  for all types of variables and any dependencies between them. This covers, in particular, the cases of discrete variables, mixed variables, and continuous variables with 
a complicated, possibly nonconnected support. While the methodology of \cite{JJP} is based on the copula functions, \cite{Nelsen}, the measure-theoretic methodology presented here is, roughly speaking, about using the properties of the Radon-Nikodym derivative to reduce the general multivariate characterization problem  to the one-dimensional conditioning  which is well understood,  \cite{Navarro}. This approach will allow us  to prove Theorem \ref{T1} and its multivariate generalization Theorem \ref{T2}, the main results of this paper.  Theorem \ref{T1} shows that conditional covariance fully determines the independence of the variables, a statement somewhat similar in spirit to the main result from \cite{Jawor2}, which states that conditional variance  characterizes an arbitrary one-dimensional probability distribution (up to a translation). Theorem \ref{T1}, a generalization of  Theorem  3.1 from \cite{JJP}, implies that for any random variables $X$ and $Y$ that are not independent, there is some rectangle $[a,b]\times [c,d]$ conditioning on which we have the nonzero Pearson's correlation coefficient. The direct conclusion is that simple Pearson's coefficient, when appropriately conditioned, is sufficient to detect even complex nonlinear relationships between the variables regardless of the problem assumptions, for the discussion  see Chapter 4 in \cite{JJP}. One of the practical advantages of statistics which are based on truncated (conditional) moments is that the doubly-truncated moments always exist,  whereas the  classical moment-based tests may fail to work properly in case of heavy-tailed distributions. For recent applications of  conditional correlation see \cite{PJPW},  \cite{JJP}, and for various applications of one-dimensional conditional moments,  see  \cite{H},\cite{JEL},\cite{Pitera}, \cite{Multi}. While the conditional  characterization of one-dimensional distributions is relatively well understood, \cite{Nada},\cite{Navarro}, the multidimensional characterization results are still difficult to handle, with the exception of special cases, \cite{MW}. Theorem \ref{T2} from Section \ref{S4}, following Theorem 3.2 from \cite{JJP}, addresses the multivariate scenario and shows that the mutual independence of random variables may be expressed in terms of conditional correlation matrix. In the proof, we use the Radon-Nikodym derivative  to reduce the dimensionality of the problem through simple induction, and we believe that similar techniques may be useful in further studies on the conditional characterisations of multivariate problems. 
Besides the discussion associated with Theorems \ref{T1} and \ref{T2}, we also address the question of when the stochastic independence can be completely characterized  by the local uncorrelation, see Remark 2  in \cite{JJP} and Chapter 6 in \cite{DK}. Our main result here is Theorem \ref{LLL} which shows that the complete characterization of  independence by local uncorrelation is possible under the appropriate assumption on the support of the joint distribution.  Finally, we mention that Theorem \ref{T1} from Section \ref{S4} fully solves Problem 6327 from \cite{AP} which, according to my knowledge, is solved only  in the special case in which the variables have a joint density, in \cite{JJP}.

This paper is organised as follows. Section \ref{S2} introduces basic notation and definitions. Section \ref{S3} presents the auxiliary results, and  Section \ref{S4} presents and proves  Theorem \ref{T1} and Theorem \ref{T2}, the main results. Section \ref{Discussion} is a discussion section which addresses various scenarios, presents corresponding numerical applications,  and discusses the concept of  independence characterization via local linear independence.

\section{Preliminaries}\label{S2}
Let $(\Omega,\Sigma,\PP)$ be a probability space. Given $A\in\Sigma$, $1_A$ will denote the characteristic function of $A$, i.e. $1_A(x)=1$ for $x\in A$ and $1_A(x)=0$ for $x\notin A$. The space $\R^n$ is naturally equipped with the sigma-algebra of Borel sets $\BB$ and  we will say that $Z\colon\Omega\to\R^n$ is a random variable (r.v.) iff $Z\colon (\Omega,\Sigma)\to (\R^n,\BB)$ is measurable, i.e. $Z^{-1}(C)\in\Sigma$ for any $C\in\BB$. The probability distribution of r.v. $Z\colon\Omega\to \R^n$ will be denoted by $\PP_Z$, i.e. $\PP_Z$ is a Borel probability measure on $\R^n$ defined by $\PP_Z(D)=\PP[Z^{-1}(D)]$. We will often use notation $\{Z\in D\}:= Z^{-1}(D)$ and $\PP[Z\in D]:=\PP_Z(D)$. If we shortly say that $Z$ is a random variable without specifying the set of values then we mean that $Z$ is one-dimensional, i.e. $Z\colon \Omega\to\R$. In this paper all random variables are defined on $(\Omega,\Sigma,\PP)$. The mathematical expectation of r.v. $Z$ is defined by $E[Z]=\int\limits_{\Omega}Zd\PP$. If $U\in\Sigma$ satisfies $\PP[U]>0$ then $\PP[\cdot|U]$ will denote the conditional probability measure given by $\PP[C|U]=\frac{\PP[C\cap U]}{\PP[U]}$, $C\in\Sigma$, and $E[Z|U]$ will denote the conditional expectation $E[Z|U]=\int\limits_{\Omega}Zd\PP[\cdot|U]$. Naturally, $E[Z|U]=\frac{E[Z\cdot 1_U]}{\PP[U]}$.  If $U\in\Sigma$ is of the form $U=\{(X_1,\dots,X_n)\in A_1,\dots,A_n\}$ then we will  write 
$$E[Z|X_1\in A_1,\dots,X_n\in A_n]:=E[Z|U].$$

 If $X$ and $Y$ are non-constant random variables with $E[X^2]<\infty$ and $E[Y^2]<\infty$ then both the covariance and the (Pearson's) correlation are well defined:
$$\cov(X,Y)=E[XY]-E[X]\cdot E[Y]\mbox{ \ and \  }\cor(X,Y)=\frac{\cov(X,Y)}{\sigma_X\cdot \sigma_Y},$$
where  $\sigma_Z=\sqrt{E[Z^2]-(E[Z])^2}$ denotes the standard deviation of r.v. $Z$.

 If $A,B\subset\R$ are Borel sets with $\PP_{(X,Y)}(A\times B)>0$ then one may consider the corresponding conditional covariance and the conditional correlation, i.e. for $U=\{(X,Y)\in A\times B\}$   one may calculate those coefficients with respect to the conditional probability measure $\PP[\cdot|U]$. The conditional covariance $\cov_{(A,B)}(X,Y)$ is defined by 
$$\cov_{(A,B)}(X,Y)
=E[XY|(X,Y)\in A\times B]-E[X|(X,Y)\in A\times B]\cdot E[Y|(X,Y)\in A\times B],$$
and the conditional correlation $\cor_{(A,B)}(X,Y)$ is defined by : 
$$\cor_{(A,B)}(X,Y)=\frac{\cov_{(A,B)}(X,Y)}{\sigma_{(A,B)}(X)\cdot \sigma_{(A,B)}(Y)},$$
where $\sigma_{(A,B)}(\cdot)$ is the standard deviation calculated with respect to the conditional probability $\PP[\cdot|U]=\PP[\cdot|X\in A, Y\in B]$.  Note that the conditional covariance is well defined for any bounded rectangle $A\times B$ regardless of whether the  $X$ and $Y$ are square-integrable.

Now, we will shortly introduce some notation and basic facts  from measure theory, \cite{Ma}, \cite{Rudin}. Let $\mcl{M}^1(\Omega,\Sigma)$ denote the set of probability measures on the measurable space $(\Omega,\Sigma)$. We will say that $Q$ is a nonnegative measure on $(\Omega,\Sigma)$ if $Q$ is of the form $Q=c\cdot P$, where $c\geq0$ and $P\in \mcl{M}^1(\Omega,\Sigma)$. We will say that $Q$ is a signed measure if $Q=c_1\cdot P_1 - c_2\cdot P_2$, where $c_1,c_2\geq0$ and $P_1,P_2\in \mcl{M}^1(\Omega,\Sigma)$. A signed measure $Q$ is absolutely continuous with respect to $P\in \mcl{M}^1(\Omega,\Sigma)$, denoted by $Q<P$, if for any $A\in\Sigma$ with $P(A)=0$ we have $Q(A)=0$. If $Q<P$, then by Radon-Nikodym theorem there exists  a measurable function $h \colon (\Omega,\Sigma)\to(\R,\mcl{B}(\R))$ satisfying:
\begin{equation}\label{RND}
Q(A)=\int\limits_A h(x)P(dx),\ A\in\Sigma.
\end{equation} 
 Any such function $h$ is called the Radon-Nikodym derivative of $Q$ with respect to $P$, denoted by $\frac{dQ}{dP}$. Naturally, if $h_1\colon \Omega\to\R$ and $h_2\colon \Omega\to\R$ satisfy \eqref{RND} then $h_1=h_2$
  $P$ - almost surely, and formally, the Radon-Nikodym derivative $\frac{dQ}{dP}$ is  an element of $L^1(\Omega,\Sigma,P)$, the space of $P-$ integrable functions, where the functions that agree $P-$almost everywhere are identified. If we shortly write $\frac{dQ}{dP}(x)=h(x)$ or $\frac{dQ}{dP}=h$, then we just  mean that function $h$ satisfies $\eqref{RND}$, and we do remember about the almost sure nature of such equalities. A useful property is that if we have three measures with $Q_2<Q_1$ and $Q_1<P$, and $\frac{dQ_2}{dP}(x)=h_2(x)$, $\frac{dQ_1}{dP}(x)=h_1(x)$,  then 
\begin{equation}\label{prop}
\frac{dQ_2}{dQ_1}(x)=\frac{h_2(x)}{h_1(x)}\cdot 1_{\{z\colon h_1(z)>0\}}(x) \ \ \mbox{ for } Q_1-\mbox{ almost any } x\in\Omega.
\end{equation}
 In  case of Borel measures on $\R^n$, there is a beautiful characterization of $\frac{dQ}{dP}$ credited to Besicovitch, namely, given $Q$ and $P$, two finite Borel measures on $\R^n$, the limit
$$h(x)=\lim\limits_{\eps\to0^+}\frac{Q(B(x,\eps))}{P(B(x,\eps))}$$
is well-defined  for $P-\mbox{ almost any } x\in\R^n$, and defines a Borel measurable function $h\colon\R^n\to\R$, where $B(x,\eps)$ denotes the Euclidean ball centered at $x$ with radius $\eps$. Most importantly, if $Q<P$ then we have
\begin{equation}\label{RNCH}
\frac{dQ}{dP}(x)=\lim\limits_{\eps\to0^+}\frac{Q(B(x,\eps))}{P(B(x,\eps))}\mbox{  \  \  \ \ \ for\ \  }P  \mbox{\ \ - almost any }\ x\in\R^n,
\end{equation}
 see  Theorem 2.12 and Chapter 2 in \cite{Ma} for more details.

\section{Auxiliary results}\label{S3}

  Paper \cite{Navarro} addresses the general problem of the characterisation of a probability distribution on $\R$ in terms of its truncated first moments. Lemma \ref{L2}  follows directly from the results of \cite{Navarro}.
\begin{lemma}\label{L2}
\begin{enumerate}
\item For any random variable $X\colon\Omega\to\R$, the function
\begin{equation}\label{C1} (a,b)\to E[X|X\in [a,b] ]
\end{equation}
uniquely determines the probability distribution $\PP_X$ (the domain of the function is the set of pairs $a<b$ with $\PP_X([a,b])>0$). 
\item For any random variable $X\colon\Omega\to\R$ with $E|X|<\infty$, the function
\begin{equation}\label{C2} 
t\to E[X|X\geq t]
\end{equation}
uniquely determines the probability distribution $\PP_X$ (the domain of the function is the set of $t\in\R$ with $\PP[X\geq t]>0$.)
\end{enumerate}
\end{lemma}
\ \\ 
 Let $\M:=\mcl{M}^1(\R,\mcl{B}(\R))$ denote the set of Borel probability measures on $\R$. In the whole paper we will assume that 
$$\KK=\{[a,b]\colon a<b\} \mbox{ or }\KK=\{[t,\infty)\colon t\in\R\}$$
although other types of intervals could be considered as well. For any $P\in\M$ let 
$$\KK^+_P=\{A\in\KK\colon P(A)>0\}.$$ 
% Paper Navarro et al. (1998) shows how  the function 
%$$ \KK^+_P \ni A\to\frac{\int_A xdP(x)}{P(A)}\in\R$$ (uniquely) determines the distribution function $F_P(t)=P((-\infty,t])$ (in case  $\KK=\{[t,\infty)\colon t\in\R\}$ it is assumed that $P$ has a first moment: $\int\limits_A|x|dP(x)<\infty$ ). 
Observation \ref{OL2} is a consequence of Lemma \ref{L2}. 

\begin{observation}\label{OL2}
   If $P_1$ and $P_2$ are two Borel probability measures on $\R$ (with\\ $\int\limits_{\R}|x|dP_i(x)<+\infty$ in case $\KK=\{[t,\infty)\colon t\in\R\}$) and $\KK^+_{P_2}\subset \KK^+_{P_1}$, then conditioning 
\begin{equation}\label{W}\frac{\int_A xdP_2(x)}{P_2(A)}=\frac{\int_A xdP_1(x)}{P_1(A)}, \ \ \ A\in \KK^+_{P_2},\end{equation}
forces $\KK^+_{P_1}=\KK^+_{P_2}$ and, by Lemma \ref{L2}, $P_2(A)=P_1(A)$ for any $A\in\B$.
\end{observation}
\begin{proof}
Let us assume that $X_1$ and $X_2$ are random variables with $\PP_{X_i}=P_i$, $i=1,2$, so we may use the convenient notation $E[X_i|X_i\in A]=\frac{\int_A x P_{i}(dx)}{P_i(A)}$, $i=1,2$. 

  \textbf{1.} Consider the case $\KK=\{[a,b]\colon a<b\}$ which does not assume the existence of the first moments. Assume for a contradiction that for some $c<d$ we have 
	$$\PP_{X_1}([c,d])>0 \mbox{ and }\PP_{X_2}([c,d])=0.$$    We have $\PP_{X_2}(\R\setminus[c,d])=1$ and without loss of generality we assume that $\PP_{X_2}((d,+\infty))>0$. Let 
\begin{equation}\label{dx}x:=\inf\{y>d\colon \PP_{X_2}([c,y])>0\}.\end{equation}
For any $h>0$ we have $\PP[X_2\in [c,x+h]]>0$ and, by \eqref{W},
$$E[X_1|X_1\in [c,x+h]\ ]=E[X_2|X_2\in [c,x+h]\ ].$$
From the assumptions  $\PP[X_2\in[c,x)]=0$, and thus $\lim\limits_{h\to0^+}E[X_2| X_2\in[c,x+h]\ ]=x$. Hence
$$E[X_1|X_1\in [c,x]\ ]=\lim\limits_{h\to0^+}E[X_1| X_1\in[c,x+h]\ ]=\lim\limits_{h\to0^+}E[X_2| X_2\in[c,x+h]\ ]=x.$$
As $E[X_1|X_1\in [c,x]\ ]=x$, we have $\PP[X_1\in[c,x)\ ]=0$. This forces $x=d$ and 
$$\PP[X_1\in [c,d]]=\PP[X_1=d]>0.$$ 
 At the same time $\PP_{X_2}(\{d\})\leq\PP_{X_2}([c,d])=0$. We will show  that  $\PP_{X_1}(\{d\})>0$ is a contradiction with $\PP_{X_2}(\{d\})=0$. By definition of $d=x$ given by \eqref{dx}, and by $\PP_{X_2}(\{d\})=0$, we have $\PP_{X_2}([d+h,d+1])>0$ for any small $h>0$. We thus may write (recall that $\KK^+_{\PP_{X_2}}\subset \KK^+_{\PP_{X_1}}$) :
$$E[X_i|X_i\in (d,d+1]\ ]=\lim\limits_{h\to0^+}E[X_i|X_i\in [d+h,d+1]\ ],\ i=1,2.$$

From the above, by \eqref{W},
\begin{equation}\label{exd}E[X_1|X_1\in (d,d+1]\ ]=E[X_2|X_2\in (d,d+1]\ ].\end{equation}
%We have $E[X_1|X_1\in [d,d+1]\ ]=E[X_2|X_2\in [d,d+1]\ ]$. Let $$ 
Let
$$M:=E[X_1|X_1\in [d,d+1]\ ]=E[X_2|X_2\in [d,d+1]\ ].$$
We have, by $\PP[X_2=d]=0$, by $\eqref{exd}$ and next by $\PP[X_1=d]>0$,
$$M=E[X_2|X_2\in[d,d+1]\ ]= E[X_2|X_2\in(d,d+1]\ ]=$$
$$=E[X_1|X_1\in(d,d+1]\ ]>E[X_1|X_1\in[d,d+1]\ ]=M, $$
a contradiction.\\
\textbf{2.} The proof of the case  $\KK=\{[t,\infty)\colon t\in\R\}$ (under the assumption $E|X_i|=\int\limits_{\R}|x|dP_i<\infty$, $i=1,2,$) is rather similar to the previous case: first we show that for
$$x_i:=\sup\{y\in\R\colon \PP[X_i\in [y,+\infty)\ ]>0\}, i=1,2, $$
the conditioning  \eqref{W} forces $x_1=x_2$. Next,  in the case $d:=x_1=x_2<+\infty$, it remains to show that the situation $\PP[X_2=d]=0$ and $\PP[X_1=d]>0$ would lead to a contradiction.
\end{proof}

 In Chapter \ref{Discussion} we will be interested in when the family $K=\{[a,b]\colon a< b\}$ may be replaced with the family of arbitrarily small intervals. Namely, we  will be interested in the conditions under which Lemma \ref{L2} and Observation \ref{OL2} hold true for arbitrary small $\eps>0$ given  given the conditioning  set 
$$K(\eps)=\{[a,b]\colon 0<b-a<\eps\}.$$
 The example below shows that  additional assumptions are necessary. 

\begin{example}\label{LCC}
Let $X_1$, $X_2$, and $Z$ be independent random variables such that $X_1$ is uniformly distributed on $[0,1]$, $X_2$ is uniformly distributed on $[2,3]$ and $Z\in\{0,1\}$ has a Bernoulli distribution $B(1,p)$, where $p\in (0,1)$. Now, we define the random variable $X=Z\cdot X_1 + (1-Z)\cdot X_2$. It is easy note that $X$ has a density function $f_p$ given by 
$$f_p(t)=p\cdot 1_{[0,1]}(t)+(1-p)\cdot 1_{[2,3]}(t).$$ Additionally, for any Borel sets
$A\subset [0,1]$ and $B\subset[2,3]$, we have
$$E[X|X\in A]=E[X_1|X_1\in A]\mbox{ and } E[X|X\in B]=E[X_2|X_2\in B].$$
Now, fix $\eps\in (0,1)$, and note that for any interval $[a,b]$ with $b-a<\eps$ the conditioning $E[X|X\in [a,b]]$ depends only on the variables $X_1$ and $X_2$, that is, it does not depend on the parameter $p\in (0,1)$. Thus, the conditioning $E[X|X\in A]$, $A\in K^+_{X}(\eps)$, does not determine the probability distribution of $X$, where:
$$ K^+_{X}(\eps)=\{A\in K(\eps)\colon \PP_X(A)>0\}\mbox{ and }\eps\in (0,1).$$
\end{example}
\bigskip
The following lemma will be used in the next section to show that the problem of the conditional uncorrelation may be reduced to the one-dimensional problem from Lemma \ref{L2} and Observation \ref{OL2}.

\begin{lemma}\label{LRN}
Let $X$ and $Z$ be random variables with $Z\geq0$, $E[Z]<\infty$ and $E[|XZ|]<\infty$. Define the nonnegative measure $Q_1$ and the signed measure $Q_2$, both on $(\R,\B)$, by: 
$$Q_1(A)=E[1_A(X)Z],\ A\in\B,\mbox{ and } Q_2(A)=E[1_A(X)XZ],\ A\in\B.$$
We have $$\frac{dQ_2}{dQ_1}(x)=x.$$
\end{lemma}
\begin{proof}
 On the measurable space $(\Omega,\Sigma_X)$, where  $\Sigma_X=\{X^{-1}(C)\colon C\in\B\}$, we define two measures $\hat{Q}_1$ and $\hat{Q}_2$ by

$$ \hat{Q_1}(X^{-1}(C))= Q_1(C)\mbox{ and } \hat{Q_2}(X^{-1}(C))=Q_2(C),\ C\in\B.$$
By the definitions of $\hat{Q_1}$ and $\hat{Q_2}$, $Q_1$ and $Q_2$, and next by the properties of the conditional expectation $E[Z|X]=E[Z|\Sigma_X],$
\begin{equation}\label{RNE1}\hat{Q_1}(X^{-1}(C))=Q_1(C)=\int\limits_{X^{-1}(C)}Zd\PP=\int\limits_{X^{-1}(C)}E[Z|\Sigma_X]d\PP\
\end{equation}
and
\begin{equation}\label{RNE2} \hat{Q_2}(X^{-1}(C))=Q_2(C)=\int\limits_{X^{-1}(C)}XZd\PP=\int\limits_{X^{-1}(C)}XE[Z|\Sigma_X]d\PP.
\end{equation}
Let $\hat{\PP}$ denote  $\PP$ restricted to $(\Omega,\Sigma_X)\subset(\Omega,\Sigma)$, i.e. $\hat{\PP}=\PP|_{(\Omega,\Sigma_X)}$. By \eqref{RNE1} and \eqref{RNE2},

\begin{equation}
\label{12}\frac{d\hat{Q_1}}{d\hat{\PP}}=E[Z|\Sigma_X] \mbox{ and } \frac{d\hat{Q_2}}{d\hat{\PP}}=XE[Z|\Sigma_X].
\end{equation}
Hence,
$$\frac{d\hat{Q_2}}{d\hat{Q_1}}=X.$$
Thus, by the definition $\hat{Q}_1$ and $\hat{Q}_2$, by the definition of the Radon-Nikodym derivative and next by standard change of variables,

$$Q_2(C)=\hat{Q_2}(X^{-1}(C))=\int\limits_{X^{-1}(C)}Xd\hat{Q_1}=\int\limits_\Omega1_C(X)Xd\hat{Q_1}=$$
$$=\int\limits_\R1_C(x)xdQ_1(x)=\int\limits_CxdQ_1(x)$$
which implies $\frac{dQ_2}{dQ_1}(x)=x$.
\end{proof}
\begin{example}
Under assumptions of Lemma \ref{LRN}, consider the special case when vector $(X,Z)$ is absolutely continuous with respect to the Lebesgue'a measure so there exists a density function $f\colon\R^2\to\R^+$ and we have: $\PP_{(X,Z)}(D)=\int\limits_D f(x,y)dydx,\ D\in \mcl{B}(\R^2).$
Then, $Q_1(A)=E[1_A(X)Z]=\int\limits_{\R^2}1_A(x)zf(x,z)dxdz=\int\limits_{A\times \R}zf(x,z)dxdz,\ \ A\in\B.$
Define
$$h(x):=\int\limits_{\R}zf(x,z)dz,$$
and note that, by Fubini's theorem,
$$Q_1(A)=\int\limits_{A}\int\limits_{\R} zf(x,z)dzdx=\int\limits_Ah(x)dx,\ \  A\in\B.$$
Similarly, by Fubini's theorem, for any $A\in\B$,
$$Q_2(A)=E[1_A(X)XZ]=\int\limits_{A\times \R}xzf(x,z)dxdz=\int\limits_{A}x\int\limits_{\R}zf(x,z)dzdx=\int\limits_{A}xh(x)dx.$$
As $Q_1(A)=\int\limits_Ah(x)dx$ and $Q_2(A)=\int\limits_Axh(x)dx$, we indeed have $\frac{dQ_2}{dQ_1}(x)=x$.
\end{example}
\section{Main results}\label{S4}
It is a simple consequence of Fubini's theorem that if random variables $X$ and $Y$ are independent (in the sense $\PP_{(X,Y)}=\PP_X\otimes\PP_Y$) then  $cov(X,Y)=0$, and it is easy to show that the independence of $X$ and $Y$  implies $cov_{(A,B)}(X,Y)=0$ for any admissible test set $A\times B$. Theorem \ref{T1}  states that this implication may be reversed if one considers the family of bounded intervals or half-bounded intervals as test sets. In particular, the square-integrable random variables $X$ and $Y$ are independent iff they are uncorrelated conditioning on the tails $\{X\geq t\}$ and $\{Y\geq s\}$, i.e. when for any $t,s$ with $\PP[X\geq t, Y\geq s]>0$ we have: $\cov_{([t,\infty),[s,\infty))}(X,Y)=0$.\\

\begin{theorem}\label{T1}
Let  $\mathcal{K}=\{[a,b]\colon a<b\}$. For any random variables $X$ and $Y$, the following conditions are equivalent:
\begin{enumerate}
\item[1)] $X$ and $Y$ are independent,
\item[2)] for any $U=\{X\in A, Y\in B\}$ with $A, B\in \mathcal{K}$ and $\PP_{(X,Y)}(A\times B)>0$,$$E[XY|U]=E[X| U]\cdot E[Y|U].$$

\end{enumerate}
Under additional assumption: $E|X|<\infty$ and $E|Y|<\infty$, the  bounded intervals may be replaced with the half-bounded intervals $\{[t,\infty)\colon\ t\in \R\}$.
\end{theorem}

\begin{proof}
%We will prove the theorem under the assumption $\KK=\{[a,b]\}_{a<b}$ or $\KK=\{[t,\infty)\}_{t\in\R}$.\\
 The implication $"1) \Rightarrow 2)"$ is a rather straigthforward calculation. We will prove the second implication $"2)\Rightarrow 1)"$.\\
\textbf{Step 1.}  Assume  that $E|X|<\infty$ and $E|Y|<\infty$, and $\KK=\{[a,b]\}_{a<b}$ or $\KK=\{[t,\infty)\}_{t\in\R}$.\\
\textbf{1a.}  First we consider the case $\PP[Y>0]=1$. Fix $B\in\KK$ with  $\PP[Y\in B]>0$ which implies that $E[Y1_B(Y)]>0$. More, for any $A\in\KK$ we have  
$$E[1_A(X)1_B(Y)]>0\Leftrightarrow E[Y1_A(X)1_B(Y)]>0.$$

From condition 2) of the theorem, for any $A\in\KK$ with $\PP_{(X,Y)}(A\times B)>0$ we have

$$\frac{E[XY1_A(X)1_B(Y)]}{\PP_{(X,Y)}(A\times B)}=\frac{E[X1_A(X)1_B(Y)]}{\PP_{(X,Y)}(A\times B)}\cdot \frac{E[Y1_A(X)1_B(Y)]}{\PP_{(X,Y)}(A\times B)}$$
and thus
\begin{equation}\label{EEE}
\frac{E[XY1_A(X)1_B(Y)]}{E[Y1_A(X)1_B(Y)]}=\frac{E[X1_A(X)1_B(Y)]}{E[1_A(X)1_B(Y)]}.
\end{equation}
 Note that 
$$\hat{P_1}(A)=E[1_A(X)1_B(Y)]\mbox{ and }\hat{Q_1}(A)= E[Y1_A(X)1_B(Y)],\ A\in\B,$$ are nonnegative Borel measures with the same class of null sets $\hat{P_1}(A)=0\Leftrightarrow\hat{Q_1}(A)=0$. We now normalize those measures to obtain the corresponding  probability measures:
$$Q_1(A):=\frac{1}{E[Y1_B(Y)]}\cdot E[Y1_A(X)1_B(Y)],\ A\in \B$$
and
$$P_1(A):=\\P[X\in A| Y\in B]=\frac{E[1_A(X)1_B(Y)]}{E[1_B(Y)]},\ A\in \B.$$
 We want to reformulate  equation \eqref{EEE} in a way which will allow us to use Lemma \ref{LRN}. In this purpose define the following signed measures
$$Q_2(A):= \frac{1}{E[Y1_B(Y)]}\cdot{E[XY1_A(X)1_B(Y)]},\ \ A\in\B,$$
$$P_2(A):=\frac{1}{E[1_B(Y)]}\cdot E[X1_A(X)1_B(Y)], \ \ A\in\B,$$
so,  for any $A\in\KK$ with $\PP[X\in A,Y\in B]>0$, equation \eqref{EEE} takes the form
\begin{equation}\label{RN}
\frac{Q_2(A)}{Q_1(A)}=\frac{P_2(A)}{P_1(A)}.
\end{equation}
Now, we apply Lemma \ref{LRN}  to the pair $Q_2$ and $Q_1$ (with $Z:=Y1_B(Y)$)  which provides  the Radon Nikodym-derivative $\frac{dQ_2}{dQ_1}(x)=x$. In case $\KK=\{[a,b]\colon a<b\}$, Equation $\eqref{RN}$ forces $\frac{dP_2}{dP_1}=\frac{dQ_2}{dQ_1}$. Indeed, by  \eqref{RNCH},  for  $(P_1+Q_1)$ - almost any $x\in\R$,
$$\frac{dP_2}{dP_1}(x)=\lim\limits_{\eps\to0^+}\frac{P_2([x-\eps,x+\eps])}{P_1([x-\eps,x+\eps])}\mbox{ and }\frac{dQ_2}{dQ_1}(x)=\lim\limits_{\eps\to0^+}\frac{Q_2([x-\eps,x+\eps])}{Q_1([x-\eps,x+\eps])},$$
and the above limits are equal to each other by \eqref{RN}. However,  in case $\KK=\{[t,\infty)\colon t\in\R\}$ we need to use Lemma \ref{LRN} again:  the pair $P_1$ and $P_2$, by Lemma \ref{LRN} with $Z:=1_B(Y)$, satisfies  $\frac{dP_2}{dP_1}(x)=x$. Hence, in both cases, the equation \eqref{RN} takes the form 
$$\frac{\int\limits_AxQ_1(dx)}{Q_1(A)}=\frac{\int\limits_AxP_1(dx)}{P_1(A)},\ \mbox{ for any } A\in\KK\mbox{ with }P_1(A)>0.$$
By Lemma \ref{L2}, we have $Q_1(A)=P_1(A)$, $A\in \B$. As $B\in\KK$ with $\PP_Y(B)>0$ was fixed arbitrarily, by definitions of $Q_1$ and $P_1$, we have obtained that for any $B\in\KK$ and  $A\in\B$ with $\PP_{(X,Y)}(A\times B)>0$, we have 
$$\frac{E[Y1_A(X)1_B(Y)]}{E[1_A(X)1_B(Y)]}=\frac{E[Y1_B(Y)]}{E[1_B(Y)]}.$$

In other words, for $A\in\B, B\in\KK$  with $\PP[X\in A, Y\in B]>0$,
\begin{equation}\label{EXY}
E[Y|X\in A,\ Y\in B]=E[Y| Y \in B].
\end{equation}
 As the conditioning $E[Y|X\in A, Y\in B]$ does not depend on $X$, it is  natural  that $X$ and $Y$ are independent. To show this, we will use Lemma \ref{LRN} again: Equation \eqref{EXY} may be rewritten as:
$$\frac{\frac{1}{\PP[X\in A]}\cdot E[Y1_A(X)1_B(Y)]}{\PP[Y\in B| X \in A]}=\frac{E[Y1_B(Y)]}{\PP[Y\in B]},$$
 and hence, by Lemma \ref{LRN}, 
$$\frac{\int\limits_B y\PP_Y[ dy| X \in A]}{\PP[Y\in B| X \in A]}=\frac{\int\limits_By\PP_Y(dy)}{\PP_Y[ B]},$$
where $A\in\B,\ B\in\KK$  are arbitrary with $\PP_{(X,Y)}(A\times B)>0$.
%\mbox{ for any }B\in \KK\mbox{ with }\PP[Y\in B|X\in A]>0.$$
%which, for any $A\in \B$ with $\PP_X(A)>0$, 
 By Observation \ref{OL2}, we have $\PP_Y(\ \cdot\ |X\in A)=\PP_Y(\cdot)$, where  $A\in\B$ is arbitrary with $\PP[X\in A]>0$. This implies that $X$ and $Y$ are independent. 

\textbf{1b.} Now consider the case $Y>-c$, where $c>0$. By the assumption, $X$ and $Y$ are conditionally uncorrelated in sense $cov_{(A,B)}(X,Y)=0$ for any admissible $A$,$B\in \KK$. By \textbf{1a.}, it is enough to show that $X$ and $Y+c$ are conditionally uncorrelated too. Note that for any $A,B\in \KK$  we have $\PP_{(X,Y+c)}(
A\times B)>0\Leftrightarrow \PP_{(X,Y)}(A\times(B-c))>0$, where $B-c=\{b-c\colon b\in B\}$. Furthermore, the covariance is translation-invariant and the class $\KK$ is translation invariant in sense $B\in\KK \Rightarrow B-c\in\KK.$ Thus, for any $A,B\in\KK$ with $\PP_{(X,Y+c)}(A\times B)>0$, we have:
$$cov_{(A,B)}(X,Y+c)=cov(X,Y+c|X\in A, Y+c\in B)=$$
$$=cov(X,Y|X\in A, Y+c\in B)=cov(X,Y|X\in A, Y\in B-c)=0.$$
 Now, as we see that $Y+c$ and $X$ are conditionally uncorrelated, by point \textbf{1a.} $X$ and $Y+c$ are independent. Hence, $X$ and $Y$ are independent.

\textbf{1c.} Now we do not assume that $Y$ is bounded from below. The random variables $X$ and $Y$ are independent if and only if for any $c\in\R$ with $\PP[Y>c]>0$ 
the $X$ and $Y$ are independent with respect to the conditional probability $\PP_c=\PP[\cdot| Y>c]$, i.e. 
$$\PP[X\in A, Y\in B|Y>c]=\PP[X\in A| Y>c]\cdot \PP[Y\in B| Y>c], \ A,B\in\B.$$
For any $c$ small enough let $cov_{\PP_c}$ denote the covariance calculated on the probability space $(\Omega,\Sigma,\PP_c)$. Note that for $A,B\in \KK$, if $B\cap [c,\infty)\neq \varnothing$ then $B\cap [c,\infty)\in\KK$. Furthermore,
$$cov_{\PP_c}(X,Y|X\in A, Y\in B)= cov(X,Y|X\in A,Y\in B \cap [c,\infty)).$$
Hence, $X$ and $Y$ are conditionally uncorrelated with respect to the conditional probability measure $\PP_c=\PP[\cdot|Y\geq c]$ which, by \textbf{1b.}, implies that $X$ and $Y$ are independent with respect to $\PP_c$. As $c$ may be arbitrarily small, $X$ and $Y$ are independent. 

\textbf{Step 2.} It remains to consider the case $\KK=\{[a,b]\colon a<b\}$ without assuming the existence of the first moments of $X$ and $Y$. To prove that $X$ and $Y$ are independent it is sufficient to prove that they are independent with respect to the conditional probability
$$\PP_{c}=\PP[\ \cdot\ |(X,Y)\in [-c,c]^2]$$
for any $c>0$ large enough to  have $\PP[(X,Y)\in [-c,c]^2]>0$.  On the probability space $(\Omega,\Sigma,\PP_c)$ the $X$ and $Y$ are integrable and conditionally uncorrelated by the assumptions (note that if $A\in \KK$ and $A\cap[-c,c]\neq\varnothing$ then $A\cap[-c,c]\in\KK$), which by \textbf{Step 1} implies that $X$ and $Y$ are $\PP_c$ - independent for any large $c$. \end{proof}
\bigskip
%\begin{remark}
%In the above proof, points \textbf{1c.} and \textbf{2.} have been  done separately because the reasoning from \textbf{Step 2.} does not apply to the class $\KK=\{[t,\infty)\colon t\in\R\}$.
%\end{remark} 
%\bigskip
\begin{remark}
One has to be careful while dealing with the quotient of measures. % of the form $\frac{P_2(D)}{P_1(D)}$, $D\in\B$. % In the above proof we could not apply the analog of Theorem 2.12 to equation \eqref{RN} in the case $K=\{(-\infty,t]\colon\ t\in\R\}$ because the quotient of the form 
For instance, in case of the half-intervals $\KK=\{(-\infty,t]\colon\ t\in\R\}$, the quotient of the form
\begin{equation}\label{RR}
\frac{P_2(A)}{P_1(A)},\  A\in \KK^+_{P_1},
\end{equation} 
does not determine neither the ratio $\frac{P_2(D)}{P_1(D)}$, $D\in\mathcal{B}(\R)$, nor even the Radon-Nikodym  derivative $\frac{dP_2}{dP_1}$. To  illustrate that, it is enough to consider a simple two-state probability space: let $\Omega=\{0,1\}$ and let $\Sigma=2^{\Omega}$ be the family of all subsets of $\Omega$. Fix $a\in (0,1)$ and define two measures on $(\Omega,\Sigma)$:
$$ P_1(\{0\})=a, P_1(\{1\})=1-a,\mbox{ and } P_2(\{0\})=\frac{a}{2}, P_2(\{1\})=1-\frac{a}{2}.$$
It is easy to notice that the ratio $\eqref{RR}$ does not depend on $a\in(0,1)$, namely, it is determined by the following:
$$\frac{P_2(\{0\})}{P_1(\{0\})}=\frac{1}{2}\mbox{ and }\frac{P_2(\{0,1\})}{P_1(\{0,1\})}=1.$$
 At the same time, the Radon-Nikodym derivative $\frac{dP_2}{dP_1}$ is uniquely defined by:
$$\frac{dP_2}{dP_1}(0)=\frac12\mbox{ and }\frac{dP_2}{dP_1}(1)=\frac{1-\frac{a}{2}}{1-a},$$
and does depend on $a\in (0,1)$.  Naturally, similar subtleties  arise in the multidimensional settings.  Replacing the family $\{[a,b]\colon a<b\}$ with the family $\{(-\infty,t]\colon t\in\R\}$ is not always straightforward  in  characterisation problems. 
\end{remark}
\bigskip
Now, we will prove the multivariate generalisation of Theorem \ref{T1}. If  $X_1,\dots,X_n$ are random variables then for $U=\{X_1\in A_1,\dots,X_n\in A_n\}$ with $\PP[U]>0$ we define
$\cov_U(X_i,X_j)$, $\cor_U(X_i,X_j)$ and $sd_U(\cdot )$ as the covariance, correlation and the standard deviation calculated with respect to the conditional probability $\PP[\cdot|U]$. Let $\Sigma_U$ denote the conditional 
correlation matrix, i.e. 
$$\Sigma_U[i,j]=cor_U(X_i,X_j),\ \ i,j=1,\dots,n.$$
Above, if for some $i$ we have $sd_U(X_i)=0$ then we  shortly write $cor_U(X_i,X_j):=\delta_{ij}$ (Kronecker delta) which will keep the notation compact. The following theorem generalizes Theorem \ref{T1} and will be proved by mathematical induction with use of the same techniques.  \\
\begin{theorem}\label{T2}
Let $X_1,\dots,X_n$ be random variables and let $\KK=\{[a,b]\colon a<b\}$.
 The following conditions are equivalent:
\begin{enumerate}
\item $X_1,\dots,X_n$ are mutually independent (i.e. $\PP_{(X_1,\dots,X_n)}=\PP_{X_1}\otimes\dots\otimes\PP_{X_n}$),
\item the conditional correlation matrix equals to the identity matrix, i.e.
$$\Sigma_U=I_n,$$
 for any $U=\{X_1\in A_1,\dots,X_n\in A_n\}$ such that $A_1,\dots,A_n\in \KK$ and $\PP[U]>0$.
\end{enumerate}
If $X_1,\dots,X_n$ have finite second moments then the bounded intervals may be replaced  with $\KK=\{[t,\infty)\colon t\in\R\}$.
\end{theorem}
\begin{proof}
The implication $"1.\Rightarrow 2."$ is straightforward and we will focus on $"2.\Rightarrow 1."$ The proof is based on the induction. The case $n=2$ is a direct conclusion from  Theorem \ref{T1}.  For the induction step assume that Theorem \ref{T2} holds true for some natural $n\geq 2$ and assume that we are given the random vector $(X_1,\dots,X_{n+1})$ which satisfies condition 2., i.e. $\Sigma_U=I_{n+1}$ for any admissible $U$.  We may  apply Theorem $\ref{T2}$ to any $n$-dimensional vector of the form $(X_i)_{i\neq j}=(X_1,\dots,X_{j-1},X_{j+1},\dots,X_n)$ which implies  the mutual independence of  $X_1,\dots,X_{j-1},X_{j+1},\dots,X_n$. Thus, to show the induction step it is enough to show that some marginal variable $X_j$ is independent of the $n$-dimensional vector $(X_i)_{i\neq j}$. The general  induction step from $"n"$ to $"n+1"$ is the same as in the case from $n=2$ to $n=3$ but the notation is slightly more complex. We will show the implication from  $"n=2"$ to $"n=3"$.  Let $(X,Y,Z):=(X_1,X_2,X_3)$. We will show that $Y$ is independent of the vector $(X,Z)$. The proof is similar the reasoning from the proof of Theorem 1 and we will be a little more concise.\\
\textbf{1.} Assume that $X,Y,Z$ are square-integrable and $K=\{[a,b]\colon a<b\}$ or $K=\{[t,\infty)\}$. \\
\textbf{1a)} Assume that $Y>0$.  If $(A,B,C)\in \KK^3$ satisfy $\PP[(X,Y,Z)\in A\times B\times C]>0$, then for $U:=\{(X,Y,Z)\in A\times B\times C\}$,  by the assumptions, we have
$$E[X\cdot Y|U]=E[X|U]\cdot E[Y|U]$$
which may be rewritten as
\begin{equation}\label{E2}\frac{E[XY1_A(X)1_B(Y)1_C(Z)]}{E[Y1_A(X)1_B(Y)1_C(Z)]}=\frac{E[X1_A(X)1_B(Y)1_C(Z)]}{E[1_A(X)1_B(Y)1_C(Z)]}.
\end{equation}
We fix $B,C\in\KK$ with $\PP[(Y,Z)\in  B\times C]>0$ and we normalize the nonnegative measures $A\to E[Y1_A(X)1_B(Y)1_C(Z)]$ and  $A\to E[1_A(X)1_B(Y)1_C(Z)]]$ to obtain the following two probability Borel measures
$$Q_1(A)=\frac{E[Y1_A(X)1_B(Y)1_C(Z)]}{E[Y1_B(Y)1_C(Z)]}\mbox{ and } P_1(A)=\frac{E[1_A(X)1_B(Y)1_C(Z)]}{E[1_B(Y)1_C(Z)]}.$$
Now we rewrite \eqref{E2} with use of $P_1$ and $Q_1$ and next, as in the proof of Theorem \ref{T1}, we use the Lemmas \ref{L2} and \ref{LRN}  to obtain $P_1=Q_1$, i.e.
\begin{equation}\label{EEEE}\frac{E[Y1_A(X)1_B(Y)1_C(Z)]}{E[Y1_B(Y)1_C(Z)]}=\frac{E[1_A(X)1_B(Y)1_C(Z)]}{E[1_B(Y)1_C(Z)]}
\end{equation}
for any $A\in\B$ and $B,C\in\KK$ with $\PP_{(Y,Z)}(B\times C)>0$. If $\PP_{(X,Y,Z)}[A\times B\times C]>0$ then we rewrite \eqref{EEEE} to obtain
\begin{equation}\label{E3}\frac{E[Y1_A(X)1_B(Y)1_C(Z)]}{E[1_A(X)1_B(Y)1_C(Z)]}=\frac{E[Y1_B(Y)1_C(Z)]]}{E[1_B(Y)1_C(Z)]}.\end{equation}
Now fix $A\in\B$ and $C\in\KK$ with $\PP_{(X,Z)}(A\times C)>0$ and normalize the nonnegative measures $B\to E[1_A(X)1_B(Y)1_C(Z)]$ and $B\to E[1_B(Y)1_C(Z)]$ from the denominators of $\eqref{E3}$, and again use Lemma \ref{LRN} and Observation $\ref{OL2}$ to obtain:
$$\frac{E[1_A(X)1_B(Y)1_C(Z)]}{1_A(X)1_C(Z)}=\frac{E[1_B(Y)1_C(Z)]]}{E[1_C(Z)]}$$
for any $A,B\in\B $ and $C\in \KK$ with $\PP_{(X,Z)}(A\times C)>0$. In other words, for such $A$,$C$ and any $B\in\B$ we have
\begin{equation}\label{YXZ}\PP[Y\in B|X\in A, Z\in C]=\PP[Y\in B|Z\in C].
\end{equation} 
By the assumptions,  $Y$ and $Z$ are conditionally uncorrelated and thus, as Theorem \ref{T2} holds true for $n=2$, $Y$ and $Z$ are independent. To sum up:  for any $B\in\B$ and any $A\in \B,C\in\KK$ with $\PP_{(X,Z)}[A\times 
C]>0$, by \eqref{YXZ} and by the independence of $Y$ and $Z$,
$$\PP[Y\in B|X\in A, Z\in C]=\PP[Y\in B|Z\in C]=\PP[Y\in B]$$
which implies that $Y$ is independent of the vector $(X,Z)$.This finishes the proof (recall that the marginals $X$ and $Z$ are independent by Theorem \ref{T2} applied to the $2$-dimensional vector $(X,Z)$). \\
\textbf{1b)} If there is $c\in\R$ such that $\PP[Y>c]=1$ then note that the random variables $X,Y+c,Z$ satisfy condition 2. of Theorem \ref{T2} and thus they are mutually independent by point 1a). This forces the independence of $X,Y,Z$. \\
\textbf{1c)} In general case $Y\in\R$, for any small $c\in\R$ with $\PP[Y>c]>0$ we consider the probability space $(\Omega,\Sigma,\PP[\cdot|Y\geq c])$ on which, as we may note, the random vector $(X,Y,Z)$ satisfy $\Sigma_U=I_3$ for any admissible $U$. By 1b), the marginals $X,Y,Z$ are mutually independent with respect to $\PP[\cdot|Y\geq c])$ for any $c$ small enough wich implies that they are independent with respect to $\PP$.\\

\textbf{2.} At the end we do not assume that $X,Y,Z$ have finite moments and we assume that $\KK=\{[a,b]\colon a<b\}$ so  the conditional standard deviations are finite. Note that  $X,Y,Z$ are bounded and conditionally uncorrelated on the probability space $(\Omega,\Sigma,\PP_c)$, where $c>0$ is large enough to have the following  conditional probability well defined 
$$\PP_c=[\ \cdot \ | X,Y,Z\in [-c,c]].$$
By point \textbf{1.}, $X,Y,Z$ are mutually independent with respect to $\PP_c$ for any such large $c$ which implies that $X,Y,Z$ are mutually independent with respect to $\PP$.
\end{proof}
By Theorem \ref{T2},  $X_1,\dots,X_n$ are mutually independent iff $\cov_U(X_i,X_j)=0$ for any pair $i\neq j$ and any $U=\{X_1\in A_1,\dots,X_n\in A_n\}$ with $\PP[U]>0$, where $A_l\in K$ for any $l=1,\dots, n$. The following remark is a straightforward conclusion.\\

\begin{remark} [Pairwise conditional independence implies mutual indepedence] Assume that $X_1,\dots,X_n$ are random variables such that for any $U=\{X_1\in A_1,\dots,X_n\in A_n\}$ with $\PP[U]>0$ and $A_i\in K, i=1,\dots,n$, we have
\begin{equation}\label{mutual}\PP[X_i \in A,\ X_j \in B|U]=\PP[X_i\in A|U]\cdot\PP[X_j\in B|U], \mbox{ for any }i\neq j.
\end{equation}
Then, $X_1,\dots,X_n$ are mutually independent. Indeed, if some $X_i$ and $X_j$ are independent with respect to the measure $\PP[\cdot|U]$, they are uncorrelated in the sense $\cov_U(X_i,X_j)=0$. Thus, By  Theorem \ref{T2}, condition \eqref{mutual} forces mutual independence.% Naturally, in case of half-intervals $\KK=\{(-\infty,t]\}$ we assume that $E|X_i|<\infty$, $i=1,\dots,n$
\end{remark}
%\begin{conclusion} 

%\end{conclusion}

\section{Discussion and numerical examples}\label{Discussion}
In this section, we discuss various scenarios and present examples in which the conditional correlation is estimated numerically. As the similar discussion  is presented in Chapter 4  in \cite{JJP}, we will focus on the situations in which the assumptions of \cite{JJP} are not satisfied. First, we will present numerical examples to illustrate the potential usefulness of the presented theory, and then we will discuss the concept  of local uncorrelation in the context of independence characterization, see Remark 2  in \cite{JJP} and Chapter 6 in \cite{DK}.\\

In practice,  given the sample of pairs of observations $(x_i,y_i)_{i=1}^n$ and a rectangle $A\times B$, the conditional versions of covariance and correlation coeficients are easy to estimate: one simply calculates the standard empirical covariance and correlation coefficients on the subsample of $(x_i,y_i)_{i=1}^n$ determined by the set of indices: 
$$I=\{i\in \{1,\dots,n\}\colon (x_i,y_i) \in A\times B \}.$$
The mathematical justification of this procedure follows directly from the properties of the conditional probability: if the random vectors $Z_i=(X_i,Y_i)$, $i=1,2,\dots$, are i.i.d. and $U=\{(X_1,Y_1)\in A\times B\}$ satisfies $\PP[U]>0$ then for 
$$N_1:=\min\{n\colon Z_n\in A\times B\}\mbox{ and } N_{k+1}:=\min\{n> N_k\colon Z_n\in A\times B\}$$ the random variables $Z_{N_1}, Z_{N_2},\dots$ are independent and have probability distribution equal to the conditional distribution $\PP_{Z_1}[\cdot |U]$, i.e.
$$\PP[Z_{N_k}\in D]=\PP[Z_1\in D|U],\ D\in\mathcal{B}(\R^2), k\in\N.$$

In practice, set $A\times B$ may be determined by the quantiles of the sample, see \cite{JEL} for the theoretical properties of such estimators, or one can choose the conditioning set based on the specific features of the 
sample. Now, we will discuss a few examples to better illustrate the general concept and its potential usefulness, without aiming to develop a rigorous  statistical framework. We will start with the simple numerical illustration of the the  
textbook example \cite{NCC}, where $X$ and $Y$ are uncorrelated but the conditional correlation varies between $0$ and $1$, depending on the choice of conditioning set. In this example the  analytic formula for  quantile conditioning can be derived explicitely, see  equation (20) in \cite{JJP}.
\begin{example}[Illustrative example.]
 Let $X$ and $Z$ be independent and such that $\PP_{X}=N(0,1)$ and $\PP[Z=1]=\PP[Z=-1]=\frac12$. Define $Y=Z\cdot X$. It is easy to see that $\PP_Y=N(0,1)$, $X$ and $Y$ are not independent, and at the same time they are uncorrelated: $\cov(X,Y)=E[YX]-E[X]E[Y]=E[ZX^2]=E[Z]\cdot E[X^2]=0.$ Also, it is easy to see that the vector $(X,Y)$ is supported on the set $\{(x,y)\in\R^2\colon |x|=|y|\}$, and that for any $t>0$ and $U=\{|X|\leq t, |Y|\leq t\}$ we will have
$$\cor_U(X,Y)=0.$$ 
The above follows directly from the geometric interpreation of the Pearson's correlation coefficient: given any square $\{|X|\leq t, |Y|\leq t\}$, the line of best fit between  $Y$ (independent variable) and $X$ (dependent variable) is horizontal. At the same time, on the set $V=\{X\geq0, Y\geq0\}$ we have full linear dependence: $Y=X$ and $\cor_V(X,Y)=1$.   Figure 1 presents a simple numerical illustration. 

%We will take a look on the behaviour of the function 
%$$h(t)=H[t,t]=cor_{((-\infty,t],(-\infty,t])}(X,Y).$$ This function may be calculated analytically but instead  we  draw the sample $(x_i,y_i)_{i=1}^n$ of size $n=10^3$  and calculate the conditional correlations numerically, see Figure \ref{fig_1} below. The conditional correlation coefficient increases until it reaches value 1 at $t=0$ as the condition $\{X\geq0,Y\geq0\}$ implies full linear dependence $X=Y$.
\begin{figure}[H]
\centering
\includegraphics[width=3in]{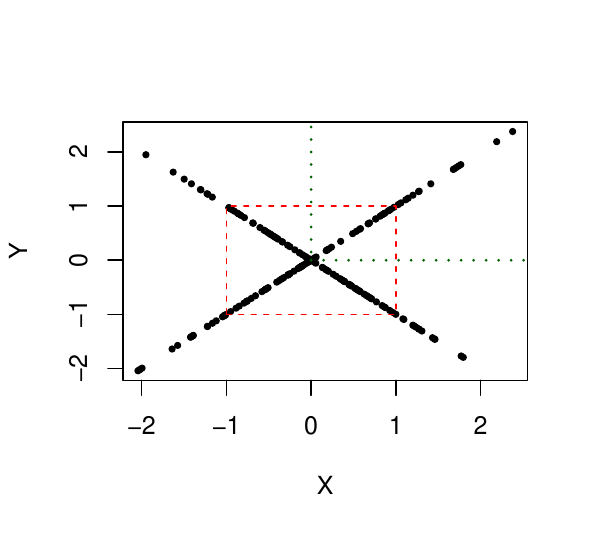}
\caption{   $200$ points are sampled from the probability distribution of the vector $(X,Y)$. The red dashed line is the contour of the square $U=[-1,1]^2$ and the green dashed line is the contour of the set $V=\{(x,y)\colon x\geq0,y\geq0\}$}
\label{fig_1}
\end{figure}
\end{example}
\begin{example}[Discrete variables] Assume that $X\in\{a_1,\dots,a_n\}\subset\R$ and $Y\in\{b_1,\dots,b_m\}\subset\R$, where $\PP[X=a_i]>0$, $\PP[Y=b_j]>0$, for all $i$, $j$. Assume additionally that $a_i<a_{i+1}$ and $b_j<b_{j+1}$ for all $i<n$, $j<n$. From Theorem \ref{T1} it follows that $X$ and $Y$ are independent if and only if for any  $i\leq n, j\leq m$.
\begin{equation}\label{Aff}
E[XY| X\leq a_i, Y\leq b_j]=E[X| X\leq a_i, Y\leq b_j]\cdot E[Y| X\leq a_i, Y\leq b_j].
\end{equation}
If $i=1$ or $j=1$, equation $\eqref{Aff}$ is satisfied trivially. Hence, in the simplest scenario: $X\in\{a_1,a_2\}$ and $Y\in\{b_1,b_2\}$, we have direct conclusion that $X$ and $Y$ are independent if and only if 
\begin{equation}\label{exy}
E[XY]=E[X]\cdot E[Y].
\end{equation}
In particular, if  $X\in\{0,1\}$ and $Y\in\{0,1\}$, it is easy to see that condition $E[XY]=E[X]\cdot E[Y]$  transforms directly into $\PP[X=1,Y=1]=\PP[X=1]\cdot\PP[Y=1],$ which implies the independence of $X$ and $Y$.

 If  $X$ and $Y$ are categorical variables then one can apply equation \eqref{Aff} to any  transformation of the variables, namely, one can apply equation \eqref{Aff} to $\hat{X}$ and $\hat{Y}$, where $\hat{X}=\varphi_1(X)$ and $\hat{Y}=\varphi_2(Y)$, and  $\varphi_1\colon\{a_1,\dots,a_n\}\to\R$ and  $\varphi_2\colon\{b_1,\dots,b_m\}\to\R$ are injective mappings chosen arbitrarilly. In practice, however, the choice of transformations $\varphi_1$ and $\varphi_2$ may influence the robustness of the corresponding statistical tests, especially when the values of $X$ and $Y$ do not have a natural order. An exception is the case of two-points distributions, in which the Pearson's coefficient will be transformations' invariant - this follows directly from its invariance under the linear transformations.
\end{example}

\begin{example}\label{Mv}[Mixed variables] Assume that  $X\in\{a_1,\dots,a_n\}\subset\R$ is a discrete variable and that $Y\in\R$ has a continuous distribution. By Theorem \ref{T1}, $X$ and $Y$ are independent if and only if for any $i\in\{1,\dots,n\}$ and $t\in\R$ with $\PP[Y\leq t]>0$,
$$E[XY| X\leq a_i, Y\leq t]=E[X| X\leq a_i, Y\leq t]\cdot E[Y| X\leq a_i, Y\leq t].$$

If $X\in\{a_1,a_2\}$, the above independence condition reduces to the following:
$$E[XY|  Y\leq t]=E[X|  Y\leq t]\cdot E[Y|Y\leq t], \mbox{ for any t with }\PP[Y\leq t]\geq0.$$

 For the illustration, Figure \ref{fig_2} shows the sample of numerically generated observations of $X$ and $Y$, where $X\in\{0,1,2\}$ is a discrete variable and $Y\in\R$ is a continuous variable, the sample size is $n=180$. The value of the empirical correlation between $X$ and $Y$  is  close to zero and equals approximately $\hat{\cor}(X,Y)=-0,01$. At the same time, if we consider the conditioning set $U=\{Y\leq 15,02\}$ which corresponds to the quantile $q_Y=\frac13$, the  conditional correlation has value $\hat{\cor}_U(X,Y)=-0,31$ and the subsample size is $n=60$. Figure \ref{fig_2} provides the graphical illustration: the black line represents the   line of the best fit (by the least square regression) between observations $Y$ and $X$,  and the red line is the estimated regression line between  $Y$ and $X$ given the conditioning set $\{(x,y)\colon y\leq 15,02\}$. The global regression line is given by the approximation $y=-0.025\cdot x + 15,86$, while  the conditional regression line is given by the equation $y=-0,38\cdot x + 14,14$. 

\begin{figure}[H]
\centering
\includegraphics[width=3in]{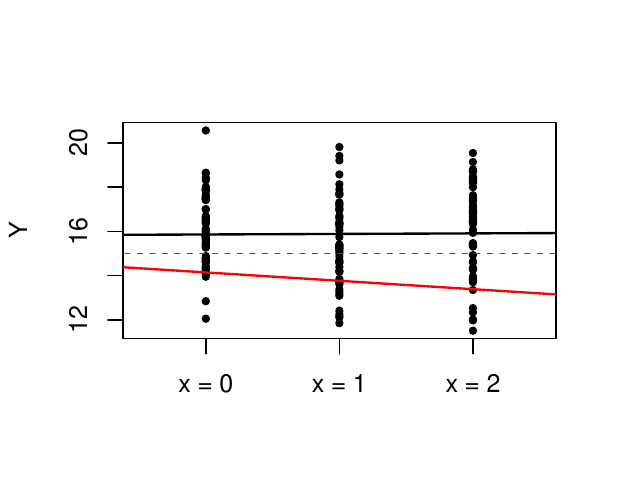}
\caption{   $180$ pairs of observations $X$ and $Y$ are sampled. The black line is the global line of linear regression between $Y$ and $X$, and the red continuous line is the line of linear regression between $Y$ and $X$ conditioning on the set determined by the quantile $q_Y=\frac13$}% The horizontal red dashed line is $y=15,02$}
\label{fig_2}
\end{figure}
\end{example}

Below, we will examine the residuals against the fitted values in the linear regression model with a continuous independent variable and  discrete dependent variables. 
\begin{example}[Linear model]
Figure \ref{fig_3} presents the residuals vs the outcome of a linear regression model $Y=aX_1+bX_2+cX_3+d+\eps$, based on the dataset Wage from ISLR package in R (the  data have been assembled from:  https://www.re3data.org/repository/r3d100011860). The sample size is $n=3\cdot 10^3$. The independent variable is worker's raw wage (a continuous variable),  and the  dependent variables are: year (in age), worker's age (in years),  and a jobclass (a categorical variable: "Industrial", "Information", coded as a $0-1$ dummy variable ). Naturally, the global empirical correlation between the residuals and the fitted values of the model equals to zero. At the same time, the conditional correlation corresponding to the left tale of the fitted values  given by the quantile $q=0.2$ equals approximately $0.26$  and indicates the  dependence between the residuals and the fitted values. For illustration, see Figure 3 which presents the data and the conditional regression line (a red line).

\begin{figure}[H]
\centering
\includegraphics[width=3in]{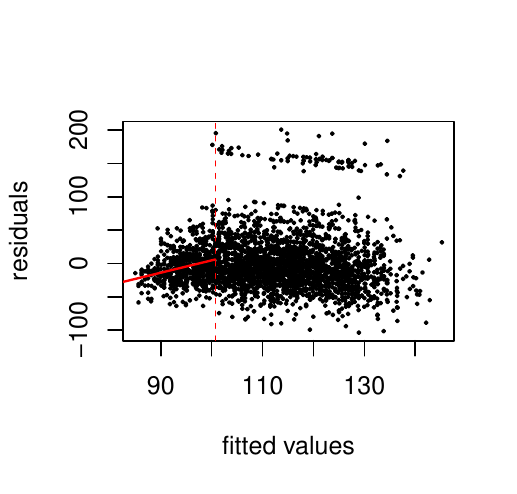}
\caption{While the global empirical correlation between the residua and the fitted values equals zero,  we observe the dependence  on the left tail of the fitted values given by the quantile $q=0.2$ }
\label{fig_3}
\end{figure} 
\end{example}

Now we will discuss the concept of local uncorrelation, \cite{DK}.  The authors of \cite{JJP}  pointed out an ellegant property of local uncorrelation: if the   random variables $X$ and $Y$ are locally linearly independent, then they are stochastically independent. More precisely,  $X$ and $Y$ are independent if and only if they are uncorrelated conditioning on rectangles of arbitrarily small size. While this observation holds true under the assumptions of \cite{JJP}, it cannot be extended to full generality. We will start the discussion with the simple example.

\begin{example}\label{SM} Assume that $X$ is supported on the discrete set $\{a_1,\dots,a_n\}$, and $Y\in\R$ is arbitrary. If $\eps>0$ is small enough then any rectangle $D=[a,b]\times [c,d]$ with $\PP_{(X,Y)}(D)>0$ and $\diam(D)<\eps$ satisfies $\PP_X([a,b])=\PP_X[\{a_j\}]$ for some $j\in\{1,\dots,n\}$,
and thus for $U=\{(X,Y)\in [a,b]\times [c,d]\}$ we have:
$$E[XY|U]=E[XY|X=a_j, Y\in [c,d]]=a_j\cdot E[Y|X=a_j, Y\in [c,d]]=$$
$$=E[X|X=a_j, Y\in [c,d]]\cdot E[Y|X=a_j, Y\in [c,d]]=E[X|U]\cdot E[Y|U],$$
regardless of whether $X$ and $Y$ are dependent or not. 
\end{example}
\ \\
Now, we will address the question of when the local uncorrelation determines independence. Recall that the support of a measure $P$, denoted by $\supp(P)$, is the smallest closed set $D$ with $P(D)=1$. In case of Example \ref{SM}, the support of $\PP_X$ is not an interval, and in such case the local conditioning $E[X|X\in A]$, $A\in K_X^+(\eps)$, does not determine the measure $\PP_X$, see Example \ref{LCC} from Section \ref{S3}. Given an arbitrary $P\in \M$ and $\eps>0$, recall the notation: $K^+_P(\eps)=\{[a,b]\colon P([a,b])>0\mbox{ and }b-a\leq\eps\}.$ We will  prove the following lemma.
\begin{lemma}\label{CCCC}
Assume that $X_1$ and $X_2$ are random variables such that both supports $\supp(\PP_{X_1})$ and $\supp(\PP_{X_2})$ are non-degenerate intervals (bounded or unbounded). Let $P_1:=\PP_{X_1}$ and $P_2:=\PP_{X_2}$. Fix $\eps>0$ and assume that  $\KK_{P_2}^+(\eps)\subset \KK_{P_1}^+(\eps)$. If
\begin{equation}\label{Epsilon}
E[X_1|X_1\in [a,b]]=E[X_2|X_2\in[a,b]] \mbox{\ \  for any \ } [a,b]\in \KK_{P_2}^+(\eps),
\end{equation}
then $P_1=P_2$.
\end{lemma}
\begin{proof}
First we will prove that equation \eqref{Epsilon} forces $\KK_{P_1}^+(\eps)=\KK_{P_2}^+(\eps)$. For a contradiction, assume that for some $a,b\in\R$ with $0<b-a\leq\eps$, we have $P_1([a,b])>0$ and $P_2([a,b])=0$. As the support of $P_2$ is an interval, this implies that $P_2([b,+\infty])=1$ oraz $P_2(-\infty,a])=1$. Without loss of generality assume that $P_2([b,+\infty])=1$.
Let
$$x:=\inf \supp(P_2).$$
As $x\geq b$, and $\supp(P_2)$ is an interval with $\supp(P_2)\subset\supp(P_1)$, we have that $x$ belongs to the interior of  $\supp(P_1)$. Hence, $P_1\left([x-\frac{\eps}{2},x)\right)>0$ and $P_1\left((x,x+\frac{\eps}{2}]\right)>0$, which forces 
\begin{equation}\label{Wik}E[X_1|X_1\in [x-\frac{\eps}{2},x+\frac{\eps}{2}]]<E[X_1|X_1\in [x,x+\frac{\eps}{2}]].
\end{equation}
At the same time, by the definition of $x$ we have  $P_2\left([x-\frac{\eps}{2},x)\right)=0$, and hence
\begin{equation}\label{Wika}E[X_2|X_2\in [x-\frac{\eps}{2},x+\frac{\eps}{2}]]=E[X_2|X_2\in [x,x+\frac{\eps}{2}]].
\end{equation}
By the assumptions: $E[X_1|X_1\in [x,x+\frac{\eps}{2}]]=E[X_2|X_2\in [x,x+\frac{\eps}{2}]]$, and thus \eqref{Wik} and \eqref{Wika} lead to: 
$$E[X_1|X_1\in [x-\frac{\eps}{2},x+\frac{\eps}{2}]]<E[X_2|X_2\in [x-\frac{\eps}{2},x+\frac{\eps}{2}]],$$
a contradiction with \eqref{Epsilon}.\\
From now we assume that $\KK_{P_2}^+(\eps)= \KK_{P_1}^+(\eps)$. Fix $A\in \KK_{P_1}^+(\eps)$ and note that both conditional measures 
$$P_1^A:=P_{1}(\ \cdot\ |A)\mbox{ and }P_2^A=P_{2}(\ \cdot\ | A),$$
 satisfy the assumptions of Lemma \ref{L2}. Namely, by \eqref{Epsilon}, for any interval $[a,b]\subset A$ ,
$$\frac{\int\limits_{[a,b]}xP_1^A(dx)}{P_1^A([a,b])}=\frac{\int\limits_{[a,b]}xP_2^A(dx)}{P_2^A([a,b])},$$
and hence, by Lemma \ref{L2}, $P_1^A=P_2^A$.  Thus, by the definition of $P_i^A$, $i=1,2$,
$$\frac{P_1([c,d])}{P_1(A)}=\frac{P_2([c,d])}{P_2(A)},\mbox{ for any }[c,d]\subset A,$$
and hence
$$\frac{P_2([c,d])}{P_1([c,d])}=\frac{P_2(A)}{P_1(A)},\mbox{ for any }[c,d]\subset A.$$
As the above quotient  $\frac{P_2([c,d])}{P_1([c,d])}$ is constant for any $[c,d]\subset A$, by \eqref{RNCH} we obtain that the Radon-Nikodym derivative $\frac{dP_2}{dP_1}$ is constant $P_1$ - almost surely on the interior of $A$.  It is easy to conclude that the R-N derivative $\frac{dP_2}{dP_1}$ is constant $P_1$-almost surely on the whole interval $A$:  write $A=[s,t]$ and note that if, for instance, $P_1(\{t\})>0$, then  the value of $\frac{dP_2}{dP_1}$ at point $t$ is defined uniquely by
$$\frac{dP_2}{dP_1}(t)=\frac{P_2(\{t\})}{P_1(\{t\})}=\lim\limits_{\eps\to0^+}\frac{P_2([t-\eps,t])}{P_1([t-\eps,t])}=\frac{P_2(A)}{P_1(A)}.$$

Now, let $x_0=0$ and $x_k:=k\cdot\frac{\eps}{2}$, $k\in\mathbb{Z}$, and let 
$$A_k:=[x_k-\frac{\eps}{2},x_k+\frac{\eps}{2}],\ k\in\mathbb{Z},$$
so the family $\{A_k\}_{k\in\mathbb{Z}}$ is the family of overlapping intervals which cover the whole real line. Let
$$\mathbb{I}=\{k\in \mathbb{Z}\colon P_1(A_k)>0 \}.$$
Naturally:  
$$ \operatorname{supp}(P_1) \subset \bigcup\limits_{i\in\mathbb{I}} A_i.$$
 As $\operatorname{supp}(P_1)$ is an interval, for any $k\in \mathbb{I}$ such that $k+1\in\mathbb{I}$ we have 
\begin{equation}\label{eheh}
P_1(A_k\cap A_{k+1})>0.
\end{equation}
 Additionally, as $A_k$ and  $A_{k+1}$ belong to $\KK_{P_1}(\eps)$,  the derivative $\frac{dP_2}{dP_1}$ is constant $P_1$- almost surely on $A_k$ and is constant $P_1$-almost surely on $A_{k+1}$. As we have $\eqref{eheh}$, by simple induction we obtain that $\frac{dP_2}{dP_1}$ is constant $P_1$-almost surely on the whole support  of $P_1$. This implies that for some $c\in\R$ we have $P_2=c\cdot P_1$. As both $P_1$ and $P_2$ are probability measures, we have $P_2=P_1$.  

\end{proof}

\begin{remark}
 A probability measure with jumps (i.e. a measure with nontrivial discrete part) still may satisfy the assumptions of Lemma \ref{CCCC}  as far as the jump points are not isolated points of the support. 
\end{remark}
\ \\
Now, we are back to the question of when the local conditioning determines the independence. We will say that $X$ and $Y$ are locally uncorrelated if and only if there is $\eps>0$ such that $\cov_U(X,Y)=0$ for any $U=\{X\in[a,b],\ Y\in[c,d]\}$ with $\PP[U]>0$ and $\diam([a,b]\times[c,d])=\max(b-a,d-c)\leq\eps$.\\

At first, note that if $P\in\mcl{M}^1(\R^2)$ has marginals $P_1$ and $P_2$, i.e.:
$$P_1(A)=P(A\times\R)\mbox{ and }P_2(A)=P(\R\times A), \ A\in\B,$$
then
$$\supp(P)\subset \supp(P_1)\times \supp(P_2).$$
Additionally, we always have:
$$\supp(P_1\otimes P_2) = \supp(P_1)\times \supp(P_2).$$
 We are ready to state the following.

%Below, the $F_{(X,Y)}$ will denote the multivariate distribution function: $F_{(X,Y)}(t,s)=\PP[X\leq t, Y\leq s]$.  Assumption \eqref{ZAL} is the two-dimensional generalization of the one-dimensional assumption from Lemma \ref{CCCC}.

%of the one-dimensional assumption on the supports $\PP_X$ and $\PP_Y$ from Lemma \ref{CCCC}.%
\begin{theorem}\label{LLL}
Assume that $X$ and $Y$ are such that both supports $\supp(\PP_{X})$ and $\supp(\PP_{Y})$ are arbitrary intervals (bounded or unbounded), and that
\begin{equation}\label{ZAL}\supp(\PP_{(X,Y)}) = \supp(\PP_{X})\times \supp(\PP_{Y}).
\end{equation}

%for any $t,s\in\R$ with $0<F_{(X,Y)}(t,s)<1$ and for any $\eps>0$, we have
%\begin{equation}\label{ZAL}
%F_{(X,Y)}(t,s)<F_{(X,Y)}(t+\eps,s)\mbox{ and }F_{(X,Y)}(t,s)<F_{(X,Y)}(t,s+\eps).
%\end{equation}
Under these assumptions, $X$ and $Y$ are independent if and only if they are locally uncorrelated.
\end{theorem}

To prove the above it is enough to repeat the steps of  the proof of Theorem \ref{T1} with the exception that instead of Lemma \ref{L2} and Observation \ref{OL2} from Section \ref{S3} we  use Lemma \ref{CCCC}. The assumptions of Theorem \ref{LLL} guarantee that all the conditional measures from the proof of Theorem \ref{T1} satisfy the assumptions of Lemma \ref{CCCC}. %Additionally, note that  the Radon-Nikodym derivative from the proof of Theorem \ref{T1} is specified by the values of measures on arbitrary small intervals, see the characterization \eqref{RNCH}.
\begin{remark}
Now we comment on the assumptions of \cite{JJP}: the assumption that the joint density of $(X,Y)$ is positive on the cube guarantees that assumption \eqref{ZAL} is satisfied, and, in particular,  the assumption that $F_X$ and $F_Y$ are invertible guarantee that the supports of $X$ and $Y$ are intervals.
\end{remark}
At the end we  present a numerical example in which $X$ and $Y$ are locally uncorrelated but not independent. In the example  the assumption \eqref{ZAL} is satisfied but  the supports of $X$ and $Y$ are not intervals. 
\begin{example}[Unconnected support]
Let us consider the sample of observation pairs $(X,Y)$, for which: $X=X_1+ 2\cdot Z_1$, and $Y=X_2+2\cdot Z_2$, where $(X_1,X_2))$ is uniformly distributed on the cube $[0,1]^2$, and $Z=(Z_1,Z_2)\in \{(0,0),(1,0),(0,1),(1,1) \}$ is not distributed uniformly: $\PP_Z(\{(0,1)\})=\PP_Z(\{(1,0)\})=\frac16$,   $\PP_Z(\{(0,0)\})=\PP_Z(\{(1,1)\})=\frac13$. Additionally, the variables $(X_1,X_2,Z_1,Z_2)$ are assumed to be mutually independent. In this toy example we treat $Z$ as a hidden variable 
that controls the location of $X$ and $Y$. Here, the local uncorrelation determines the  independence with respect to the conditional measure $\PP_{(X,Y)}(\ \cdot\ | Z=(a,b))$ for  any $a,b\in\{0,1\}$, whereas full independence $\PP_{(X,Y)}=\PP_X\otimes \PP_Y$ is not satisfied.

\begin{figure}[H]
\centering
\includegraphics[width=3in]{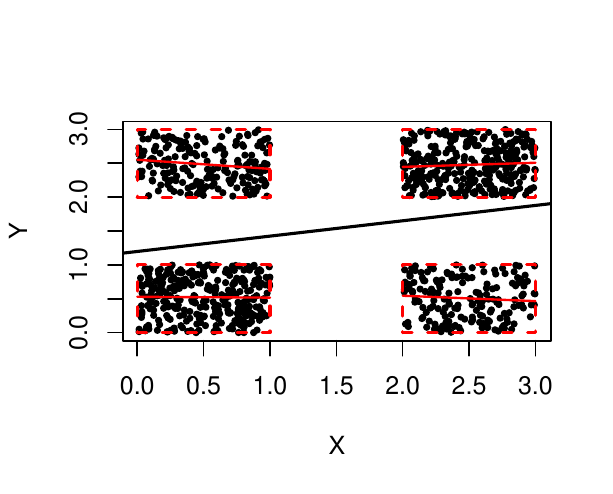}
\caption{  While $X$ and $Y$ are visibly dependent, they are independent locally. The red lines are conditional regression lines and the black line is the global regression line, the sample size $n=10^3$.}
\label{fig_5}
\end{figure} 
\end{example}

\textbf{Declaration of interests}. I have nothing to declare.\\
\textbf{Funding}. This research did not receive any specific grant from funding agencies in the public, commercial, or not-for-profit sectors.
%\textbf{Data availability statement}. No datasets were generated or analysed during the current study.\\

 %\printbibliography

\end{document}